\documentclass[12pt]{amsart}
\usepackage{amsmath,amssymb,amsthm}
\usepackage{setspace}
\usepackage{moreverb}
\usepackage{mathrsfs}
\usepackage{url}
\usepackage[all]{xy}
\usepackage{lscape}
\usepackage{tikz}
\usetikzlibrary{arrows}
\usepackage{color}

\usepackage[utf8]{inputenc}
\usepackage[english]{babel}
\usepackage{amsmath, amsthm, amsfonts, amssymb}
\usepackage{graphicx}
\usepackage{color}
\usepackage{verbatim}
\usepackage{geometry}
\usepackage{mathtools} 
\usepackage{array}

\makeatletter
\def\zz\ignorespaces{\@ifnextchar-{}{\phantom{-}}}
\newcolumntype{C}{>{\zz}{c}}
\makeatother

\newcommand{\ignore}[1]{}

\newtheorem{theorem}{Theorem}[section]
\newtheorem{lemma}[theorem]{Lemma}
\newtheorem{corollary}[theorem]{Corollary}
\newtheorem{proposition}[theorem]{Proposition}
\newtheorem{algorithm}[theorem]{Algorithm}
\theoremstyle{definition}
\newtheorem{definition}[theorem]{Definition}
\newtheorem{example}[theorem]{Example}

\newtheorem{ass}[theorem]{Assumption}

\theoremstyle{remark}
\newtheorem{remark}[theorem]{Remark}

\numberwithin{equation}{section}

\input xy
\xyoption{all}

\newcommand{\bC}{\mathbb{C}}

\newcommand{\bZ}{{\mathbb{Z}}}

\definecolor{grey}{rgb}{0.75,0.75,0.75}
\definecolor{orange}{rgb}{1.0,0.5,0.5}
\definecolor{brown}{rgb}{0.5,0.25,0.0}
\definecolor{pink}{rgb}{1.0,0.5,0.5}

\newcommand{\fM}{{\mathfrak m}}

\newcommand{\fa}{\mathfrak{a}}

\newcommand{\cN}{{\mathcal N}}
\newcommand{\cO}{{\mathcal O}}

\newcommand{\K}{\mathcal{K}}
\newcommand{\Z}{\mathbb{Z}}

\newcommand{\C}{\mathbb{C}}

\newcommand{\m}{\mathfrak{m}}

\usepackage[colorlinks=true, linkcolor=blue, citecolor=green, urlcolor=cyan, pagebackref]{hyperref}

\begin{document}

\title[Base points of ideals]{Effective computation of base points of ideals in two-dimensional local
rings}

\author[M. Alberich-Carrami\~nana]{Maria Alberich-Carrami\~nana}

\author[J. \`Alvarez Montaner]{Josep \`Alvarez Montaner}

\author[G. Blanco]{Guillem Blanco }

\address{Departament de Matem\`atiques\\
Univ. Polit\`ecnica de Catalunya\\ Av. Diagonal 647, Barcelona
08028, Spain} \email{Maria.Alberich@upc.edu, Josep.Alvarez@upc.edu,
gblanco92@gmail.com}

\thanks{ All three authors are supported by Spanish Ministerio de
Econom\'ia y Competitividad MTM2015-69135-P. MAC and JAM are also
supported by Generalitat de Catalunya 2014SGR-634 project and they are with the Barcelona Graduate School of Mathematics
(BGSMath). MAC is
also with the Institut de Rob\`otica i Inform\`atica Industrial
(CSIC-UPC)}



\begin{abstract}
We provide an algorithm that allows to describe the minimal log-resolution of an
ideal in a smooth complex surface from the minimal log-resolution of its generators.
In order to make this algorithm effective we present a modified version of the
Newton-Puiseux algorithm that allows to compute the Puiseux decomposition of a product
of not necessarily reduced or irreducible elements together with their algebraic
multiplicity in each factor.

\end{abstract}

\maketitle

\section{Introduction} \label{sec1}

Let $(X,O)$ be a germ of smooth complex surface and  $\cO_{X,O}$ the
ring of germs of holomorphic functions in a neighborhood of $O$,
which we identify with $\C\{x,y\}$ by taking local coordinates.
Let $\fa \subseteq \cO_{X}$ be an ideal sheaf. From now on, if no confusion arises, we will indistinctly denote
 by $\fa$  the sheaf ideal or its stalk at $O$. In this later case
 we will be considering an ideal $\fa \subseteq \C\{x,y\}$.
 We also denote $\fM=\fM_{X,O}\subseteq \C\{x,y\}$ the maximal ideal.

\vskip 2mm

A {\it
log-resolution} of the pair $(X,\fa)$, or a log-resolution of $\fa$
for short, is a proper birational morphism $\pi: X' \rightarrow X$
such that $X'$ is smooth,  the preimage of $\fa$ is locally
principal, that is $\fa\cdot\cO_{X'} = \cO_{X'}\left(-F\right)$ for
some effective Cartier divisor $F$, and $F+E$ is a divisor with
simple normal crossings support where $E = Exc\left(\pi\right)$ is
the exceptional locus.
We point out that $F$ decomposes into its affine and exceptional part $F=F_{\rm aff} + F_{\rm exc}$
according to its support. In particular  $F= F_{\rm exc}$ when $\fa$ is $\fM$-primary.

\vskip 2mm

In order to describe the divisor $F$ we will use the theory
of {\it weighted clusters} developed by Casas-Alvero in \cite{Cas00}.
Namely, any log-resolution  is a composition of blow-ups of
points infinitely near to $O$. Hence, attached to $F$, there is a
pair $\K=(K, v)$ where $K$ is the set of infinitely near points that
have been blown-up and $v:K \longrightarrow \bZ$ is a valuation map that encodes the
coefficients of the exceptional components in $F$. If $E_i$
is the exceptional divisor that arises from the blowing-up of a
point $p_i$, we have $F_{\rm exc}=\sum_i d_i E_i$ where $v(p_i)=d_i$.
The {\it weighted cluster of base points }
 $BP(\fa)=(B, \beta)$ of an ideal $\fa$ is the weighted cluster associated to the \emph{minimal}
 log-resolution of $\fa$. Indeed, the version that we present in this work is a mild generalization
of the construction given in \cite{Cas00} for the case of $\fM$-primary ideals (see Remark \ref{base_point_ideal}).
Roughly speaking, it is the cluster of base points of $\fa$ weighted by the values of generic
elements of $\fa$.

\vskip 2mm

The aim of this work is
to provide an algorithm that describes the minimal log-resolution of any ideal $\fa$
or equivalently, its associated { weighted cluster of base points } $BP(\fa)$.
If $\fa=(f)$ is a principal ideal with $f\in \C\{x,y\}$, the minimal
log-resolution of $\fa$ equals the minimal log-resolution of the
reduced curve $\xi_{\rm red}$ of $\xi: f=0$. Indeed, computer
algebra systems such as {\tt Singular} \cite{DGPS} or {\tt Magma} \cite{Magma} can
compute the divisor $F$ whenever $f$ is reduced.

\vskip 2mm

If the ideal $\fa=(a_1,\dots,a_r)\subseteq \C\{x,y\}$ is not principal,
 the minimal log-resolution $\pi$ of $\fa$ is no longer straightforwardly deduced from the
minimal log-resolutions $\pi_i: X'_i \longrightarrow X$ of the
principal ideals $\fa_i=(a_i)$ corresponding to each generator.
Neither $\pi$ dominates any $\pi_i$, nor the minimal proper
birational morphism $\pi':Y\longrightarrow X$ dominating all
$\pi_i$, which is the minimal log-resolution of the principal ideal
$(a_1\cdots a_r)$, dominates $\pi$.
Clearly, $\pi, \pi'$ and any $\pi_i$ factor through the blow-up of
the {\it infinitely near points} that are common to all curves $a_i=0$. Apart from
this, no other inclusion between infinitely near points attached to
$\fa$ and the principal ideals $\fa_i$ hold. In this work we will describe the minimal log-resolution of $\fa$ from the minimal log-resolution of its generators. We will not
only provide the infinitely near points which must be blown-up and blown-down to
reach $\pi$ from those of $\pi_i$, but we will also describe the
divisor $F$ in terms of the divisors $F_i$, with $\fa_i\cdot\cO_{X'}
= \cO_{X'}\left(-F_i\right)$.

\vskip 2mm

The structure of the paper is as follows. In Section \ref{sec2} we review the basics on the theory
of weighted clusters. In particular, we introduce the weighted cluster of base points $BP(\fa)$ associated
to any ideal $\fa=(a_1,\dots,a_r) \subseteq \bC\{x,y\}$. Our definition is a mild generalization of the weighted
cluster defined by Casas-Alvero in \cite[\S7.2]{Cas00} for the case of $\fM$-primary ideals. Indeed, we have
a decomposition $\fa= (g) \cdot \fa'$, where $g\in \bC\{x,y\}$ is the greatest common divisor of the generators of $\fa$
and $\fa'$ is $\fM$-primary. From the weighted cluster $BP(\fa')$ and the cluster of {\it singular} points
of the reduced germ $\eta_{\rm red}$ of $\eta: g=0$ we can describe $BP(\fa)$ (see Section \ref{cluster_log}).

\vskip 2mm

In Section \ref{sec3} we provide an
alternative description of $BP(\fa)$ that will be more useful for our purposes. The advantage
is that the {\it virtual values} of the weighted cluster depend on the values of the generators of the ideal.
In Section \ref{sec4} we present the main result of this paper. Namely, we provide an algorithm
(see Algorithm \ref{A1}) that allows to compute $BP(\fa)$ for any given ideal $\fa$.
The idea behind our method is to give a first approximation of $BP(\fa)$ by means of a
weighted cluster associated to the product of the generators of the ideal $\fa'$ and $g$.
Then, using some technical results developed in Section \ref{sec41}, we construct some intermediate
weighted clusters that lead to the desired result.

\vskip 2mm

In Section \ref{sec5} we provide a generalization of the Newton-Puiseux algorithm that makes Algorithm \ref{A1} effective. The main feature of this version is that, given a set of
elements $f_1, \dots, f_r \in \mathbb{C}\{x, y\}$ not
necessarily reduced or irreducible, it computes the Puiseux decomposition of the product
$f = f_1 \cdots f_r$, that is, the Puiseux series of $f$ along with their algebraic
multiplicities in each of the factors $f_1, \dots, f_r$. Indeed, our method  provides all the information
needed to recover both the decomposition of each factor and the decomposition of the whole product at the same time.
One of the key ingredients is to use the \emph{square-free factorization} of $f$.

\vskip 2mm

Finally, we would like to mention that a log-resolution of an ideal is a sort of principalization.
An alternative approach to the problem of principalization was given by Cassou-Nogu\`es and Veys in \cite{CNV14}.
They describe an algorithm that transforms an ideal $\fa=(a_1,\dots,a_r) \in \bC\{x,y\}$ into
a principal one by means of {\it Newton maps}. At the $i$-th  step of the algorithm, the Newton map is
determined from the initial forms of the transformed generators at the $(i-1)$-th step.
However, no connection is shown between the transformations of the principal ideals $(a_i)$ defined
by each generator and the transformation of the whole ideal $\fa$. In contrast, our approach provides such a relationship. Furthermore, we not only describe the divisor $F$ from the divisors $F_i$, but our procedure also allows a {\it topological} generalization avoiding the requirement of explicit generators: given the equisingularity class of the generators $a_i$ and the contacts between any pair of branches of different generators, we explicitly describe the equisingularity class of a generic element of $\fa$.

\section{Preliminaries} \label{sec2}

Let $X$ be a smooth complex surface and $\cO_{X,O}\cong \C\{x,y\}$
the ring of germs of holomorphic functions in a neighborhood of a
smooth point $O\in X$. Consider  a
sequence of blow-ups above $O$,
$${\displaystyle \pi:X'=X_{r+1}\xrightarrow{}X_{r}\xrightarrow{}\cdots \xrightarrow{}X_1=X},$$
with  $X_{i+1}={\rm Bl}_{p_i} X_i$ for a point $p_i\in X_i$ blowing-down to $O\in X$.

\vskip 2mm

Let ${\rm Div}(X')$ be the group of integral divisors in $X'$, i.e.
divisors of the form $D = \sum_i d_i E_i$ where the $E_i$ are
pairwise different (non necessarily exceptional) prime divisors and
$d_i \in \bZ$. Among them, we will consider divisors in the lattice
$\Lambda_\pi:= \bZ E_1 \oplus \cdots \oplus \bZ E_r$ of exceptional
divisors and we will simply refer them as divisors with {\it
exceptional support}. We have two different basis of this $\bZ$-module
given by the {\it total transforms} and the {\it strict transforms}
of the exceptional components. For simplicity, we will also denote
the strict transforms by $E_i$ and the total transforms by $\overline{E_i}$.
In particular, any divisor $D_{exc}\in  \Lambda_\pi$ can be presented in
two different ways
$$D_{\rm exc} =\sum_{i=1}^r v_i E_i =\sum_{i=1}^r e_i \overline{E_i},$$
where the weights $v_i$ (resp. $e_i$) are the {\it values} (resp.
{\it multiplicities}) of $D_{\rm exc}$. In general, any divisor  $ D\in {\rm
Div}(X')$ has a decomposition $D=D_{\rm exc} + D_{\rm aff}$  into
its  {\it exceptional} and  {\it affine} part according to its
support.

\vskip 2mm

The relation between  values and multiplicities is given by the
combinatorics of the configuration of exceptional divisors. To such
purpose we will use the theory of {\it infinitely near points}.
Roughly speaking, the exceptional divisors will be parameterized by
the sequence of points $p_i \in E_i \subseteq X_{i+1}$ that are
blown-up to achieve $\pi:X' \longrightarrow X$.

\vskip 2mm

A point $p$  infinitely near  to
the origin $O$ is a point lying on the
exceptional divisor of the composition  of a finite sequence of blow-ups.
The set $\cN_O$ of points $p$  infinitely near to $O$ can be viewed as the
disjoint union of the exceptional divisors appearing at successive
blow-ups above $O$.  The points in $X$ will be called \emph{proper}
points in order to distinguish them from the infinitely near ones. The set
$\cN_O$ is endowed with a partial order relation~$\leq$ defined by $p
\leq q$ if and only if $q$ is infinitely near to $p$. In this case
we will say that $p$ \emph{precedes} $q$.

\vskip 2mm

Given two points $p \leq q$ infinitely near to $O$, we say that
$q$ is \emph{proximate} to $p$ if and only if $q$ belongs to the
exceptional divisor $E_p$ as proper or infinitely near point. We
will denote this relation by $q\rightarrow p$. By construction, an
infinitely near point $q$ is proximate to just one or two points. In
the former case we say that $q$ is a \emph{free point}, in the later
it is a \emph{satellite point}.

\vskip 2mm

Let $p_1=O$ and $p_2, \dots, p_r $ be the infinitely near points that
appear in the successive blow-ups composing $\pi:X' \longrightarrow X$.
Let $D=\sum_{i=1}^r e_i \overline{E_i} =\sum_{i=1}^r v_i E_i$ be a divisor
with exceptional support. Then, we have the following relation:
\begin{equation} \label{eq:proximity}
v_{i}= e_i + \sum_{p_i \rightarrow p_j} v_{j}.
\end{equation}
Indeed, we can encode all the proximity relations in
the {\it proximity matrix} ${P}=(p_{i,j})$ defined as:
$$p_{i,j}:= \left\{
\begin{array}{rcl}
 1 & \hskip 1mm & \mbox{if } j=i, \\
-1 & \hskip 1mm & \mbox{if } \ p_i\rightarrow p_j, \\
 0 & \hskip 1mm & \mbox{otherwise.}
\end{array} \right.
$$
Then, the vectors   ${\bf e}=(e_1,\dots, e_n)$  and ${\bf v}=(v_1,\dots, v_n)$ of  multiplicities
and values satisfy the base change formula ${\bf e}^\top ={P} \cdot {\bf v}^\top$.

\begin{remark}
 We may also consider the {\it intersection matrix } $N=(n_{i,j})$ defined as:
$$n_{i,j}:= \left\{
\begin{array}{l}
1  \hskip 2cm  \mbox{if } E_i\cap E_j\neq \emptyset, \ \ i\neq j, \\
0 \hskip 2cm   \mbox{if } E_i\cap E_j = \emptyset, \ \ i\neq j, \\
-r_i-1 \hskip .8cm \mbox{if } i=j.
\end{array} \right.
$$ where $r_i$ is the number of points  proximate to $p_i$. Then, $N=-P^\top P$.
\end{remark}

\subsection{Weighted clusters of infinitely near points} \label{sec21}
In order to describe divisors with exceptional support we will use
the theory of weighted clusters of infinitely near points developed
by Casas-Alvero in \cite{Cas00}. In the sequel, we will fix all the
basic notions that we will use  but we encourage the interested
reader to take a look at \cite{Cas00} in order to get a deeper
insight.

\vskip 2mm

A \emph{cluster} is a finite subset $K \subsetneq \cN_O$  of infinitely near points to the origin $O$ such that, if $p \in K$, then any preceding point $q < p$ also belongs to $K$.
A \emph{weighted cluster} $\K=(K,e)$ is a cluster $K$, called the {\it underlying cluster} of $\K$,  together with a
map $e:K\longrightarrow\Z$, where $e_p:= e(p)$ is the  \emph{virtual multiplicity} of $\K$ at $p$.
Alternatively, we may also weight a cluster by a system of \emph{virtual values} $v_p:= v(p)$  given by a map
$v:K\longrightarrow\Z$. Both multiplicities and values are related recursively by
\begin{equation} \label{eq:mult_val}
v_{p}= e_p + \sum_{p \rightarrow q} v_{q},
\end{equation}
and the reason for maintaining this apparent redundancy will become clear after Section \ref{cluster_log}.
\emph{Consistent clusters} are those weighted clusters with non-negative excesses,
where the \emph{excess} of $\K$ at $p_i$ is
$$\rho_p=e_p -\sum_{{q\rightarrow p}}e_q.$$
We say that it is \emph{strictly consistent} if, furthermore, $e_p>0$ for all $p\in K$.

\vskip 2mm

We define the sum of two weighted clusters $\K=(K,e)$ and $\K'=(K',e')$ weighted by multiplicities
as $\K + \K'=(K\cup K',e + e')$ extending $e_p=0$ for $p \not\in K'\backslash K$ and
$e'_p=0$ for $p \not\in K\backslash K'$.

\vskip 2mm

Let $\pi_{K}: X' \longrightarrow X$ be the sequence of blow-ups centered at the points of the weighted cluster $\K$.
Let $\xi: f=0$ be a germ of plane curve defined by $f \in \cO_{X,O}$. The \emph{total transform} of
$\xi$ is the pull-back $\overline{\xi}:=\pi_K^*\xi$ and
can be written as\footnote{We use the notation $\overline{E_p}$ and $E_p$ to emphasize that the exceptional component corresponds to
the point $p\in K$.}
$$\overline{\xi}= \widetilde{\xi} +
\sum_{p \in K} e_p (f) \overline{E_p}=
\widetilde{\xi} + \sum_{p \in K} v_p (f) E_p,$$ where $e_p (f)$ (resp. $v_p
(f)$) is the multiplicity (resp. value) of $\xi$ at $p$, and
$\widetilde{\xi}$ is the \emph{strict transform} that can be also
realized as the closure of $\widetilde{\xi}:=\pi^{-1}(\xi-\{O\})$. Neither $e_p (f)$ nor $v_p
(f)$ depend on the equation $f$ defining the curve $\xi$. The {\it virtual
transform} with respect to the weighted cluster $\K$ is defined as
$$\check{\xi} := \widetilde{\xi} + \sum_{p \in K} (v_p (f)-v_p) E_p.$$
If $v_p (f)\geq v_p$, for all $ p\in K$ then, we say that the curve $\xi$ {\it goes virtually}
through $\K$. It {\it goes sharply} through $\K$ when $v_p (f)= v_p$, for all $p\in K$ and $\xi$ has no singular, see Section \ref{cluster_log}, points outside those in $\K$.

\vskip 2mm

Let $\K_{\leq p}$ be the weighted subcluster consisting on all the points preceding a given point $p\in K$
with the same weights as $\K$. If we denote by $\check{\xi}_p$ the virtual transform with respect to $\K_{\leq p}$,
then, for all $p\in K$, we have
\begin{equation} \label{eq:virtual_p}
e_p(\check{\xi}_p)=v_p(f) - \sum_{p\to q} v_q.
\end{equation}

\subsection{The weighted cluster associated to the minimal log-resolution} \label{cluster_log}
We have a correspondence between clusters of infinitely near points to the origin $O\in X$ and divisors
with exceptional support in $X'$, where  $\pi: X' \rightarrow X$ is a sequence
of blow-ups centered at the points of the cluster. Namely, we have a correspondence
between a cluster $\K=(K,v)$ and a divisor $D_\K=\sum_{p\in K} v_p {E_p} \in \Lambda_\pi$.

\vskip 2mm

In the case that $\pi: X' \rightarrow X$ is the minimal log-resolution
of an ideal $\fa \subseteq \C\{x,y\}$ we are going to describe the weighted cluster associated to the
exceptional part of the Cartier divisor $F$ such that $\fa\cdot\cO_{X'} = \cO_{X'}\left(-F\right)$.

\vskip 2mm

$\bullet$ {\bf The case of reduced plane curves:} Consider a principal ideal
generated by  a reduced function $f\in \C\{x,y\}$. It defines a germ
of curve $\xi: f= 0$ at $O$, whose {\it branches} are the germs
defined by the irreducible factors of $f$.

\vskip 2mm

The affine part $F_{\rm aff}$ of the divisor $F$ corresponds to the
strict transform $\widetilde{\xi}$ of the curve. The weighted
cluster associated to the divisor $F_{\rm exc}$ will be denoted
$\mathcal{S}(\xi)$ and coincides with the weighted cluster defined in \cite[\S
3.8]{Cas00}. Namely, consider the set $\cN_O(\xi)$ of points $p\in\cN_O$ infinitely
near to the origin $O$ \emph{lying on} $\xi$, i.e. those points such
that $e_p(f) >0$. Such a point is \emph{simple} (resp.
\emph{multiple}) if $e_p(f) =1$ (resp. $e_p(f) >1$). It is
\emph{singular} if it is either multiple, satellite or precedes a
satellite point lying on $\xi$. Then, $\mathcal{S}(\xi)$ is the weighted
cluster of singular points, weighted by the multiplicities, or the
values, of $\xi$. It is a strictly consistent cluster since it satisfies the
\emph{proximity equalities} (see \cite[3.5.3]{Cas00}).
$$e_{p}(f)= \sum_{q \rightarrow p} e_{q}(f).$$

\vskip 2mm

$\bullet$ {\bf The case of ideals:} Given an ideal $\fa=(a_1,\dots,a_r) \subseteq \bC\{x,y\}$
we have a decomposition  $\fa= (g) \cdot \fa'$, where $g\in \bC\{x,y\}$  is the greatest common
divisor of the generators of $\fa$ and  $\fa' = (f_1, \dots, f_r)$ is $\fM$-primary.

\vskip 2mm

In the case that $\fa$ is indeed $\fM$-primary we have that the divisor $F=F_{\rm
exc}$ only has exceptional support.  Its associated weighted cluster
is the weighted cluster of {\it base points} $BP(\fa)$ defined in
\cite[\S7.2]{Cas00}. It consists of the points shared by the curves
defined by generic elements of $\fa$  weighted by the corresponding
multiplicities or values.

\vskip 2mm

In the general case, we have a decomposition $F=F_{\rm aff}+F_{\rm
exc}$, so we will have to treat the affine and the exceptional
part of the divisor $F$ separately.  To describe the affine part, consider a decomposition of $g$
into its irreducible factors $g= g_1^{a_1} \cdots g_t^{a_t}$. Then, we have
$F_{\rm aff} = a_1 E_1 + \cdots + a_t E_t$, where $E_i$ is the strict transform on $X'$ of
the irreducible germ $\eta_i: g_i = 0$.
In order to describe the weighted cluster associated to the
exceptional part $F_{\rm exc}$ we will give the following  natural generalization
of the cluster of base points considered in \cite[\S7.2]{Cas00}:

\vskip 2mm

Define the multiplicity of the ideal at $O$ as $e _O(\fa)= e = \min
\{e_O(f) \, | \,  f \in \fa\}$. If $p \in E_O$ is any
point in the first neighborhood of $O$, the pull-back of functions
induces an injective homomorphism of rings $\varphi_p:
\mathcal{O}_{X,O} \longrightarrow \mathcal{O}_{X_p,p}$. We consider
the ideal defined as $\fa_p=z^{-e} \varphi_p(\fa)$ (where $z$ is
any equation for the exceptional divisor $E_O$ near $p$), that is,
generated by the germs $\check{\xi}_p : z^{-e} \varphi_p(f)=0$, i.e.
the virtual transforms of $\xi : f=0$, $f \in \fa$. These
definitions can be extended to any $p \in \mathcal{N}_O$, as well as
the multiplicity $e_p(\fa) = e_p(\fa_p)$. Since the germs $\xi :
f=0$, $f \in \fa$, have no fixed part outside $\eta: g=0$, the set of
points such that $0 < e_p(\fa) \not =  e_p(g)$   is finite, and
hence it is a cluster. We then add to this weighted cluster the weighted cluster $\mathcal{S}(\eta_{\rm red})$ of
singular points of the reduced  germ $\eta_{\rm red}$, (weighted by
the multiplicities of the ideal $\fa$ in those points) giving a larger cluster that
we also refer as the  {\em weighted cluster of base points} of $\fa$ and
denote by $BP(\fa)$. Observe that $BP(\fa)$ does not depend on the
decomposition $\fa=(g)\cdot \fa'$, namely on the unit affecting
either $g$ or  the  generators of $\fa'$.

\vskip 2mm

\begin{remark}\label{base_point_ideal}
Let  $BP(\fa')= (K', e')$ be the cluster of base points
of the $\fM$-primary ideal $\fa'$ and let $\mathcal{S}(\eta_{\rm
red})= (S, e'')$ be the cluster of singular points of the reduced germ associated to $\eta: g=0$. Then,
the cluster $BP(\fa) = (B, e)$ can be also described as follows:

\begin{itemize}
\item[$\cdot$] $B=K' \cup S$.
\item[$\cdot$] $e_p= e'_p+ e_p(g)$, extending $e'_p=0$ for $p$ outside $K'$.
\end{itemize}

In particular, we have to consider the multiplicities $e_p(g)$ of the non-reduced
germ $\eta$ instead of the multiplicities $e_p''$ of  $\eta_{\rm red}$. In fact, $BP(\fa)=BP((g)) + BP(\fa')$.
If we weight $BP(\fa) = (B, \beta)$ and  $BP(\fa')= (K', v')$ using values, then
we have $\beta_p= v'_p+ v_p(g)$.

\vskip 2mm

We point out that, in the case of $\fM$-primary ideals, this description coincides
with the one given in \cite[\S7.2]{Cas00}. Hence, we
recover easily the properties in loc.cit. also in this
case. Namely, $BP(\fa)$ is strictly consistent and all germs $\xi : f=0$, $f \in \fa$, go virtually
through $BP(\fa)$. Generic germs go sharply through it, they miss
any fixed finite set of points not in $BP(\fa) \cup \mathcal{S}(\eta_{\rm red})$, in particular
have the same equisingularity class, and their fixed part reduces to $g$, having no multiple factors
outside those of $g$. Furthermore, the ideal $\fa$ may be generated by a finite system of generators
defining germs going sharply through $BP(\fa)$.
\end{remark}

\vskip 2mm

\section{An alternative characterization of the weighted cluster of base points of an ideal} \label{sec3}
The aim of this section  is to give an alternative characterization of the weighted cluster
of base points  $BP(\fa)$ associated to any ideal $\fa \subseteq \bC\{x,y\}$.
This new approach will be more suitable for our purposes in the rest of this work. For the sake
of simplicity in the notations, we will assume that  $\fa =(f_1, \dots, f_r)$ is $\fM$-primary
(see Remark \ref{base_point_ideal}). Then, we will describe a weighted cluster $\mathcal{K} =(K, v)$
with virtual values $v$ depending  on the values of the curves $\xi_i: f_i = 0, i = 1, \dots, r$ and we will then prove that it  is equal to $BP(\fa)$.

\vskip 2mm

\begin{definition}\label{cluster2}
Let $\fa = (f_1, \dots, f_r) \subseteq \bC\{x,y\}$ be an $\m$-primary ideal. For any point $p\in \cN_O$ equal or infinitely near to $O$
we define the value  $v_{p} := \min\{v_p(f_1), \dots, v_p(f_r)\}, $ and recursively on the proximate points, we define
$h_{O} = 0$ and $$h_{p}:= \sum _{p
\rightarrow q} v_{q}.$$
Define also the weighted cluster $\mathcal{K} =(K, v)$, where
 $K$ is the set of points $p \in \cN_O$  such
that $h_{p} < v_{p}$. The corresponding virtual multiplicities are defined
as $e_p=v_p - h_p$, for each $p \in K$.
\end{definition}

\vskip 2mm

To be a proper weighted cluster, we need to check that $K$ is finite
and that for any $p\in K$, all the preceding points also belong to $K$.
To do so, we will start with a technical lemma.

\begin{lemma} \label{lemma:lemma0}
Let $\fa = (f_1, \dots, f_r) \subseteq \bC\{x,y\}$ be an $\m$-primary ideal. Then, for any
 $f \in \fa$ and any $p\in \cN_O$, we have
$ v_p(f) \geq \min\{v_p(f_1), \dots, v_p(f_r)\}.$
\end{lemma}
\begin{proof}
Assume that $f = g_1 f_1 + \cdots + g_r f_r,$ for $g_1, \dots, g_r
\in  \bC\{x,y\}.$ Using the fact that  $v_{p}(\cdot)$ is a discrete valuation in
$ \bC\{x,y\}$ (see \cite[\S 4.5]{Cas00}), we have:
\begin{align*}
v_p(f) = &\ v_p(g_1 f_1 + \cdots + g_r f_r) \geq \min \{ v_p(g_1 f_1), \dots, v_p(g_r f_r) \} \nonumber \\
       = & \min_i \{ v_p(g_i) + v_p(f_i) \} \geq \min \{ v_p(f_1), \dots, v_p(f_r) \},
\end{align*}
where in the last inequality we used that $v_p(g) \geq 0, \forall g
\in  \bC\{x,y\}$.
\end{proof}

Next, we  prove that the definition of the weighted cluster
$\mathcal{K}=(K,v)$ does not depend on the generators of the ideal.

\begin{lemma} \label{lemma:generators}
Let $\fa = (f_1, \dots, f_r) \subseteq \bC\{x,y\}$ be an $\m$-primary ideal.
The virtual values of the weighted cluster
$\mathcal{K}=(K,v)$ associated to $\fa$ do not depend on the generators of the ideal.
In other words, $v_p = \min_{f \in \fa} \{v_p(f)\}$.
\end{lemma}
\begin{proof}
It is clear that $v_p = \min_i \{v_p(f_i)\} \geq \min_{f \in \fa} \{
v_p(f)\}$ since $\{f_1, \dots, f_r\} \subset \fa$. On the other hand,
by Lemma \ref{lemma:lemma0}, $v_p(f) \geq \min_i\{v_p(f_i)\},$ for all $f
\in \fa$, hence $\min_{ f \in \fa}\{v_p(f)\} \geq v_p$ and the result
follows.
\end{proof}

\begin{lemma} \label{lemma:lemma1}
Under the assumptions of Definition \ref{cluster2}, the inequality
$h_{p}\leq v_{p}$ holds for any point $p\in \cN_O$ equal or infinitely near to $O$.
\end{lemma}
\begin{proof}
The inequality is clear when $p=O$. Now, assume that $p$ is free,
so it is proximate to one point $p \rightarrow q$. Then, we have:
\begin{equation*}
h_{p} = v_{q} = \min_i \{v_{q}(f_i)\} \leq \min_i \{ v_{p}(f_i) \} = v_{p}.
\end{equation*}
If $p$ is satellite, it is proximate to two points $p \rightarrow q$ and $p \rightarrow q'$. Then, we have:
\begin{eqnarray*}
h_{p} = v_{q} + v_{q'} & = & \min_i \{v_{q}(f_i)\}+ \min_i \{v_{q'}(f_i)\} \\
 & \leq & \min_i \{v_{q}(f_i)+v_{q'}(f_i)\}= \min_i \{ v_{p}(f_i) \} = v_{p}.
\end{eqnarray*}
\end{proof}

\begin{lemma} \label{lemma:lemma2}
Let $\fa = (f_1, \dots, f_r) \subseteq \cO_{X,O}$ be an $\m$-primary ideal.
If there exists a generator $f_i$ such that $h_{p} = v_{p}
=v_{p}(f_i)$, then we have $e_{p}(f_i)=0$ and $v_{q} = v_{q}(f_i)$ for
any point $q\in \cN_O$ such that $p \rightarrow q$.
\end{lemma}

\begin{proof}
If $p$ is a free point, we take the unique point $q$ such that $p \rightarrow q$. Notice that
$$v_{p}(f_i) = v_{p} = h_{p} = v_{q} = \min_j \{ v_{q}(f_j)\},$$
so we have $v_{p}(f_i) \leq v_{q}(f_i)$. It follows from Equation \ref{eq:mult_val}  that
$v_{p}(f_i) = v_{q}(f_i)$ and  $e_{p}(f_i) = 0$, hence, $v_{q} =
v_{q}(f_i)$.

If $p$ is satellite, we take the points $q$ and $q'$ such that $p \rightarrow q$, $p \rightarrow q'$.
We have $$v_{p}(f_i) = v_{p} = h_{p} = v_{q} +
v_{q'},$$ thus $v_{p}(f_i) \leq v_{q}(f_i) +
v_{q'}(f_i)$. Using Equation \ref{eq:mult_val}, we obtain
$v_{p}(f_i) = v_{q}(f_i) + v_{q'}(f_i)$ and $e_{p}(f_i)=0$.
Finally, if $v_{q} < v_{q}(f_i)$ or $v_{q'} < v_{q'}(f_i)$,
then $h_{p} = v_{q} + v_{q'} < v_{p}(f_i)$, so we get a contradiction.
\end{proof}

\begin{proposition} \label{prop:proposition3}
Under the assumptions of Definition \ref{cluster2}, if $p \in K$, then any point $q$ preceding $p$ also belongs to $K$.
\end{proposition}
\begin{proof}
We will prove the converse statement: assume that $q \notin K$, i.e. $h_{q}=v_{q}$. We will prove
$h_{p}=v_{p}$ for any $p$ in the first neighborhood of $q$, and it will follow inductively
 $h_{p}=v_{p}$, i.e. $p \notin K$, for any point $p$
infinitely near to $q$.

Assume that $q \notin K$ and let $p$ be a point in the first neighborhood of $q$, in particular $p \rightarrow q$.
Consider a generator $f_i$ such that $v_{q}= \min_j \{ v_{q}(f_j)\}= v_q(f_i)$, hence
$h_{q}=v_{q}=v_{q}(f_i)$. If $p$ is satellite, we take the second point $q'$ such that  $p \rightarrow q'$.
Then, by Lemma \ref{lemma:lemma2}
\begin{equation*}
h_{p}=v_{q}+v_{q'}= v_{q}(f_i)+ v_{q'}(f_i)= v_{p}(f_i),
\end{equation*}
and by Lemma \ref{lemma:lemma1}, $h_{p}=v_{p}(f_i)= v_{p}$.
If $p$ is free, the same reasoning is valid by taking $v_{q'} = v_{q'}(f_i) = 0$.
\end{proof}

We will show next that $\K=(K,v)$ equals the weighted cluster $BP(\fa)=(B,\beta)$ of base points of $\fa$
and we will conclude that $K$ is finite.

\begin{proposition} \label{prop:proposition4}
Let $BP(\fa)=(B,\beta)$ be the weighted cluster of base points of a \mbox{$\m$-primary} ideal
$\fa = (f_1, \dots, f_r) \subseteq  \bC\{x,y\}$ and $\K=(K,v)$ as given in Definition \ref{cluster2}. Let $p \in B$, then $p \in K$ and the equality of virtual values $\beta_{p}= v_{p}$ is satisfied.
\end{proposition}

\begin{proof}
Let $f=g_1f_1+\cdots + g_r f_r$ be the equation of a germ $\xi : f=0$ going sharply through $BP(\fa)$.
From Lemma \ref{lemma:lemma0} we obtain
\begin{equation*}
\beta_p = \ v_{p}(f) \geq \min_i \{v_{p}(f_i)\} = v_p.
\end{equation*}
Now, by \cite[7.2.16]{Cas00}, we may find a system of generators $\fa=(h_1,
\dots, h_s)$ such that $\zeta_i : h_i = 0$ goes sharply through
$BP(\fa)$. Then,
\begin{equation*}
v_p = \min_i \{v_{p}(f_i)\} \geq \min_i \{v_{p}(h_i)\}  = \beta_p,
\end{equation*}
after applying Lemma \ref{lemma:lemma0}, once again, to the elements $f_i$ expressed in terms of $h_1,
\dots, h_s$. Therefore, the equality $v_p = \beta_p$ follows.

Since the same equality holds for all the points preceding $p$, we
infer
\begin{equation*}
v_p - h_p = \beta_p - \sum_{p \rightarrow q} \beta_q = b_p > 0,
\end{equation*}
i.e. $v_p > h_p$, since $b_p$ is the virtual multiplicity at $p$
of the strictly consistent weighted cluster $BP(\fa)$.
\end{proof}

\begin{theorem}\label{thm:theorem5}
The weighted clusters $BP(\fa)=(B,\beta)$ and $\mathcal{K}=(K,v)$ are equal.
In particular, $K$ is finite.
\end{theorem}

\begin{proof}
From Proposition \ref{prop:proposition4} we already have $B\subseteq K$.
We will prove the other inclusion using induction on the order of
neighborhood which a point $p\in K$ belongs to.

For $p = O$, it is clear that $p$
belongs to both $K$ and $B$. Now, assume that the assertion is true for all the points
preceding $p$, which are in $K$ by  Proposition \ref{prop:proposition3}.
Let $q \in B$ be the antecessor of $p$. By \cite[7.2.6]{Cas00}, $p
\in B$ if and only if $0 < \min _{f \in \fa}
\{e_{p}(\check{\xi} _{p})\}$, where $e_{p} (
\check{\xi}_{p})$ is the virtual multiplicity of the germ
$\xi: f=0$ at $p$ relative to the weighted cluster $BP(\fa)
_{\leq p}$. This is equivalent, by Equation \ref{eq:virtual_p}, to
\begin{equation*}
\min_{f \in \fa} \{v_{p}(f)\} > \sum_{p \rightarrow s} \beta_{s}.
\end{equation*}

By Lemma \ref{lemma:generators}, $v_p = \min_{f \in
\fa}\{v_p(f)\}.$ Thus, applying Proposition \ref{prop:proposition4} to
the points preceding $p$, we have that $p$ belongs to $B$ if and only if
\begin{equation*}
v_p > \sum_{p \rightarrow s} \beta_s =  \sum_{p \rightarrow s} v_s =
h_p.
\end{equation*}
\end{proof}

\begin{remark}
Theorem \ref{thm:theorem5} is a generalization for $\fM$-primary ideals of \cite[2.5]{AC04}
which describes the base points of a pencil $\lambda_1 f_1 + \lambda_2 f_2$ of curves $f_1,f_2 \in \bC\{x,y\}$,
$\lambda_1, \lambda_2 \in \bC$.
\end{remark}

\begin{corollary} \label{corollary:cor-thm-5}
Let $BP(\fa)=(B,\beta)$ be the weighted cluster of base points of a $\fM$-primary ideal $\fa = (f_1, \dots, f_r) \subseteq  \bC\{x,y\}$.
Any cluster of infinitely near points $K'$ weighted by the values
$$v(p) = v_p = \min_i\{v_p(f_i)\}, \hskip 3mm \forall p \in K',$$ or,
alternatively, weighted by the multiplicities  $$e(p) = e_p = v_p
- \sum_{p \rightarrow q} v_q, \hskip 3mm \forall p \in K',$$ satisfying $e_p
\neq 0$ for any $p \in K'$ is a subcluster of $B$.
\end{corollary}
\begin{proof}
Since, by definition, $K = \{ p \in \mathcal{N}_O\ |\ e_p > 0 \}$,
clearly $K' \subseteq K$. Then, the result follows using Theorem \ref{thm:theorem5}.
\end{proof}

\section{An algorithm to compute the base points of an ideal} \label{sec4}
In this section we will provide an algorithm that allows to
compute the weighted cluster of base points of any ideal $\fa=(a_1,\dots,a_r) \subseteq \bC\{x,y\}$.
First we recall that we have a decomposition
$\fa= (g) \cdot \fa'$, where $g\in \bC\{x,y\}$  is the greatest common
divisor of the generators of $\fa$ and  $\fa' = (f_1, \dots, f_r)$ is $\fM$-primary.
Moreover, the weighted cluster of base points of $\fa$ is described in terms of the base points of $\fa'$
and the cluster of singular points of the reduced germ $\eta_{\rm red}$ of  $\eta: g=0$ (see Remark \ref{base_point_ideal}).

\vskip 2mm

The cluster $\mathcal{S}(\eta_{\rm red})$ is easy to describe using the Newton-Puiseux algorithm (see \cite{Cas00})
so the bulk of the process is in the computation of $BP(\fa')$. Before proceeding to describe the algorithm we need
to introduce several technical results that will allow us to compute  the weighted cluster
$BP(\fa')$ in terms of the weighted cluster of singular points of the germs defined by the set of generators
 $\xi_i: f_i = 0, i = 1, \dots, r$.

\subsection{Adding free and satellite points} \label{sec41}
For the sake of simplicity, we will assume throughout this subsection that our ideal  $\fa = (f_1, \dots, f_r) \subseteq  \bC\{x,y\}$ is $\fM$-primary.
In order to compute the weighted cluster $BP(\fa)=(B,\beta)$ we will start with a weighted cluster associated to
the product of the generators  $\xi: f_1 \cdots f_r = 0$ which gives a first approximation.
Then, using the results of this section, we will add the necessary free and satellite points
to this weighted cluster to obtain all the base points.

\vskip 2mm

\begin{ass} \label{assumption}
Assume $\mathcal{K}' = (K', v)$ is a weighted cluster with a system of
values $v_p = \min_i \{v_p(f_i)\}$ for any $p \in K'$ and satisfying that any point $p \in B$ singular for $\xi_{\rm red}$ is already in $K'$.
\end{ass}


\vskip 2mm

The following set of technical results will allow us to decide which
points we need to add to $K'$  in order to obtain all the base points.
The first result states that all the free points in $BP(\fa)$ lie on
the generators.


\vskip 2mm

\begin{lemma} \label{lemma:free-points-generators}
Let $q$ be a free point that does not lie on any curve $\xi_i: f_i=0, \hskip 2mm i=1,\dots,r$ then, $q
\not\in BP(\fa)$.
\end{lemma}
\begin{proof}
Let $q \rightarrow p$. If $q \not\in \xi_i: f_i=0, $ for all $i$, then
$v_{q}(f_i) = v_p(f_i)$ for all $i = 1, \dots, r$ and $v_{q} =
\min_i\{v_{q}(f_i) \} = \min_i\{v_p(f_i)\} = v_p$, hence $e_{q}
= v_{q} - v_{p} = 0$ and $q \not\in BP(\fa)$.
\end{proof}

The next result characterizes the free points in $BP(\fa)$ that are not singular for the reduced germ $\xi_{\rm red}$ associated to the generators.

\begin{proposition} \label{prop:missing-free-points}
Let $\mathcal{K}'$ be a weighted cluster as in Assumption \ref{assumption}. Let $q \not\in K'$ be a free
point proximate to $p \in K'$. Then, $q$ is in $BP(\fa)$ if and only if
any generator $f_i$ with $v_p(f_i)=v_p$ satisfies $e_q(f_i)>0$.
\end{proposition}

\begin{proof}
We shall apply Theorem \ref{thm:theorem5} to characterize whether $q$
belongs to $BP(\fa)$. By definition, $v_{q} = \min_j \{v_{q}(f_j)\}$
and $v_{q}(f_j) = e_{q}(f_j) + v_p(f_j),$ for $j = 1,\dots, r$.

\vskip 2mm

Set $\Lambda:=\{ j \hskip 2mm | \hskip 2mm e_q(f_j) >0 \}$. Comparing
$$v_q= \min_{j\in \Lambda, k \not\in \Lambda} \{v_p(f_k), v_p(f_j) +e_q(f_j)\} \hskip 5mm {\rm and} \hskip 5mm
v_p= \min_i \{v_p(f_i)\},$$ we infer $e_q=v_q-v_p >0$ if and only if $v_p(f_k)> v_p$, $\forall k \not\in \Lambda$
which is equivalent to $\{ i \hskip 2mm | \hskip 2mm v_p(f_i) =v_p \} \subseteq \Lambda$.
\end{proof}

\begin{remark}
Under the hypothesis of Proposition \ref{prop:missing-free-points} we observe that there might be two
generators, say $f_i,f_j$, such that $e_q(f_i)>0$, $e_q(f_j)>0$ although $q$ is not singular for the reduced
germ of $f_1\cdots f_r=0$. This may happen when $f_i$ and $f_j$ have a common factor which is not a common factor
of the rest of generators. This is a subtle difference with respect to the case of pencils treated in \cite{AC04}.
\end{remark}

Our next result deals with the
satellite base points not already in $K'$. Notice that these missing
satellite points will not lie on any generator, otherwise they would
belong to the singular points of $\xi_{\rm red}$.

\begin{proposition} \label{prop:missing-sat-points}
Let $\mathcal{K}'$ be a weighted cluster as in Assumption \ref{assumption}. Let $q \not\in K'$ be a satellite
point proximate to $p, p' \in K'$. Then, $q$ is in $BP(\fa)$ if and
only if for each generator $f_i$ either $v_p(f_i) > v_p$ or
$v_{p'}(f_i) > v_{p'}$.
\end{proposition}

\begin{proof}
Let us start by proving the converse implication. We know that $q
\not\in f_j$ for any $j = 1, \dots, r$, otherwise $q$ would be in $K'$.
We want to see that $e_{q} = v_{q} - v_p - v_{p'} > 0$. Then, $v_{q}
= \min_j\{v_{q}(f_j)\} = \min_j\{v_p(f_j) + v_{p'}(f_j)\}$ and the last
equality is true because $e_{q}(f_j) = 0$. By hypothesis, $v_p(f_j)
+ v_{p'}(f_j) > v_p + v_{p'}$, for any $j$, hence $v_{q} > v_p + v_{p'}$ as we
wanted.

For the other implication, let us assume the contrary, that is,
there exists a generator $f_i$ such that $v_p(f_i) = v_p$ and
$v_{p'}(f_i) = v_{p'}$. We know that $q \not\in f_i$, otherwise it
would be in $K'$. By definition, $v_{q} = \min_j\{v_{q}(f_j)\} =
\min_j\{v_p(f_j) + v_{p'}(f_j)\} = v_p(f_i) + v_{p'}(f_i) = v_p +
v_{p'}$, implying that $e_{q} = 0$, which is a contradiction with
the fact that $q \in BP(\fa)$.
\end{proof}

\subsection{An algorithm to compute the base points of an ideal} \label{sec42}
In this subsection we go back to our original setup, so let
$\fa=(a_1,\dots,a_r) \subseteq \bC\{x,y\}$ be an ideal that admits
 a decomposition $\fa= (g) \cdot \fa'$, where $g\in \bC\{x,y\}$  is the greatest common
divisor of the generators of $\fa$ and  $\fa' = (f_1, \dots, f_r)$ is $\fM$-primary.
With all the technical results stated above and the relation between $BP(\fa)$, $BP(\fa')$ and the cluster
$\mathcal{S}(\eta_{\rm red})$ of singular points of $\eta: g=0$ (see Remark \ref{base_point_ideal}), we present our algorithm.

\begin{algorithm} { (Base points of an ideal)} \label{A1}

\vskip 2mm

\noindent {\tt Input:}  An ideal  $\fa=(a_1,\dots,a_r) \subseteq \bC\{x,y\}$.

\noindent {\tt Output:} The weighted cluster of base points $BP(\fa)$.

\vskip 2mm

\begin{enumerate}
  \item Find $g=\gcd(a_1,\dots,a_r)$ and set $a_i=gf_i$. Compute $f=gf_1 \cdots f_r$.

  \vskip 2mm

  \item Find the cluster $\overline{K}$ of singular points of the reduced germ $\xi_{\rm red}$,  where $\xi: f=0$,
  and the system of virtual values $\{v_p(f_i)\}_{p\in\overline{K}}$, and $\{v_p(g)\}_{p\in\overline{K}}$.

  \vskip 2mm

  \noindent Compute $v_p = \min_{i}\{v_p(f_i)\}$ for $p \in \overline{K}$ and set the weighted cluster
  $\overline{\mathcal{K}}= (\overline{K},v)$.

  \vskip 2mm

  \item Define $\mathcal{K}' = (K', v)$ from $\overline{\mathcal{K}}$ by adding, if necessary, the missing free points using Proposition \ref{prop:missing-free-points} weighted by the values
  $ v_p = \min_i \{v_p(f_i)\}$ for each new $p \in K' \setminus \overline{K}$.

  \vskip 2mm

  \item Define $\mathcal{K''} = (K'', v)$ from $\mathcal{K}'$ by adding, if necessary, the missing satellite points using Proposition \ref{prop:missing-sat-points}, weighted by the values $ v_p = \min_i \{ v_p(f_i) \}$ for each new $p \in K'' \setminus K'$.

  \vskip 2mm

  \item Compute, recursively on the order of neighborhood $p$ belongs to, the multiplicities
  $e_p = v_p - \sum_{p \rightarrow q} v_q.$

  \vskip 2mm

  \noindent Define $\mathcal{K} = (K, v)$ with $K \subset K''$ containing the points
  $p \in \mathcal{K}$ such that $e_p \neq 0$ and the virtual values $ v_p$.

  \vskip 2mm

  \item Set $BP(\fa')=\mathcal{K}$, the weighted cluster of base points of $\fa'=(f_1,\dots,f_r)$.

   \vskip 2mm

  \item From $\overline{K}$ extract the cluster of singular points $\mathcal{S}(\eta_{\rm red}) =(S,v'')$,
  where $\eta:g=0$.

  \vskip 2mm

  \item Return $BP(\fa)= (B,\beta)$, where $B= K\cup S$ and $\beta_p=v_p + v_p(g)$, $\forall p\in B$.
\end{enumerate}
\end{algorithm}

Our next result proves the correctness of the Algorithm.

\begin{theorem} \label{thm:algorithm}
Algorithm \ref{A1} computes $BP(\fa) = (B, \beta)$, the weighted cluster of base points of the ideal $\fa$.
\end{theorem}

\begin{proof}
Since the cluster $\overline{K}$ fulfills the hypothesis of Proposition \ref{prop:missing-free-points} and Proposition \ref{prop:missing-sat-points}, we can use them to add the remaining base points.

After step ${\rm iv)}$ all the base points have been added. Indeed,  if we had to add a
missing base point in the first neighborhood of a point already in
$K'$, it would have to be free as we have added all the missing
satellites in the last step. This free point would have to lie on a
generator, by Lemma \ref{lemma:free-points-generators}, and it would have
to be after one of the new satellite points, otherwise we would have
added it in the fourth step. But that is impossible because the new
satellite points cannot lie on a generator, by Assumption \ref{assumption}, and hence, neither can do any of its successors.

By Corollary \ref{corollary:cor-thm-5}, after removing the points $p$
in $K''$ such that $e_p = 0$, the resulting cluster $K$ is inside $B$ and since no base point is missing it must be equal to $B$.
\end{proof}

It is a classical result that the equisingularity class, and equivalently the topological class \cite{zariski-equisingularity}, \cite{brauer}, of any element $f = f_1^{\alpha_1} \cdots f_r^{\alpha_r} \in \mathcal{O}_{X, O}$ with $f_1, \dots, f_r$ irreducible, is determined by the equisingularity class of each component $f_i^{\alpha_i}$ together with the intersection multiplicities\footnote{Noether's intersection formula states that $[f\cdot g]_{O}=\sum_{p\in K} e_p(f)e_p(g)$.}  $[f_i\cdot f_j]_{O}$, for all $i\neq j$. This can be described by means of the proximity matrix associated to the log-resolution of $(f)$ and its vector of multiplicities. As a corollary of Theorem \ref{thm:algorithm} we obtain the following generalization.

\begin{corollary}
Given an ideal $\mathfrak{a} = (a_1, \dots, a_r) \subseteq \mathcal{O}_{X, O}$, the equisingular
class of a generic element of $\mathfrak{a}$ is determined by the equisingularity class of each generator $a_i$, and the intersection multiplicities of every pair of branches from different generators.
\end{corollary}
\begin{proof}
Generic elements in $\mathfrak{a}$ go sharply through the weighted cluster of base points
$BP(\mathfrak{a})$. By Theorem \ref{thm:algorithm}, the relative position of the infinitely near
points in $BP(\mathfrak{a})$ and the multiplicities, or values, are completely determined by
the equisingularity class of each $a_i$ and the intersection multiplicities between any pair of branches
of different generators.
\end{proof}

\begin{example}
  Consider the ideal $$\mathfrak{a} = (a_1, a_2, a_3) = \big((y^5 + x^7)^2  + y^{10}x, x^8(y^3 + x^5), y^8(y^2 - x^3)\big) \subseteq \mathbb{C}\{x, y\}.$$
  The steps of Algorithm \ref{A1} are performed as follows:

  \vskip 2mm

\begin{enumerate}
  \item  We have that $g = \gcd(a_1, a_2, a_3) = 1$, so the ideal is $\fM$-primary. Then, we compute the product
  of the generators $f = a_1 a_2 a_3$.
  \vskip 2mm
  \item The cluster $\overline{K}$ of singular points of $\xi_{\textrm{red}}$ with $\xi : f = 0$ is described by means
  of the proximity matrix:

\setlength{\arraycolsep}{0.0pt}
\setcounter{MaxMatrixCols}{21}
    \[
    P_{\overline{K}} = \;
  {\tiny
       \begin{bmatrix*}[C]
        1 & 0 & 0 & 0 & 0 & 0 & 0 & 0 & 0 & 0 & 0 & 0 & 0 & 0 & 0\\
       -1 & 1 & 0 & 0 & 0 & 0 & 0 & 0 & 0 & 0 & 0 & 0 & 0 & 0 & 0\\
       -1 & 0 & 1 & 0 & 0 & 0 & 0 & 0 & 0 & 0 & 0 & 0 & 0 & 0 & 0\\
        0 & 0 & -1 & 1 & 0 & 0 & 0 & 0 & 0 & 0 & 0 & 0 & 0 & 0 & 0\\
       -1 & 0 & -1 & 0 & 1 & 0 & 0 & 0 & 0 & 0 & 0 & 0 & 0 & 0 & 0\\
       -1 & 0 & 0 & 0 & -1 & 1 & 0 & 0 & 0 & 0 & 0 & 0 & 0 & 0 & 0\\
        0 & 0 & 0 & 0 & -1 & -1 & 1 & 0 & 0 & 0 & 0 & 0 & 0 & 0 & 0\\
        0 & 0 & 0 & 0 & 0 & 0 & -1 & 1 & 0 & 0 & 0 & 0 & 0 & 0 & 0\\
        0 & 0 & 0 & 0 & 0 & 0 & 0 & -1 & 1 & 0 & 0 & 0 & 0 & 0 & 0\\
        0 & 0 & 0 & 0 & 0 & 0 & 0 & 0 & -1 & 1 & 0 & 0 & 0 & 0 & 0\\
        0 & 0 & 0 & 0 & 0 & 0 & 0 & 0 & -1 & -1 & 1 & 0 & 0 & 0 & 0\\
        0 & 0 & 0 & 0 & 0 & 0 & 0 & 0 & 0 & 0 & -1 & 1 & 0 & 0 & 0\\
        0 & 0 & 0 & 0 & -1 & 0 & 0 & 0 & 0 & 0 & 0 & 0 & 1 & 0 & 0\\
        0 & 0 & -1 & 0 & -1 & 0 & 0 & 0 & 0 & 0 & 0 & 0 & 0 & 1 & 0\\
        0 & 0 & 0 & 0 & 0 & 0 & 0 & 0 & 0 & 0 & 0 & 0 & 0 & -1 & 1\\
       \end{bmatrix*}}.
     \]
  The  virtual values $\{v_p(a_i)\}_{p \in \overline{K}}, i = 1, 2, 3 $ are the following:
     \[
        v(a_1) = {\scriptstyle
        [10\ 10\  14\ 14\ 28\ 40\  70\ 72\  74\  75\ 150\ 151\ 28\ 42\ 42]^{T},
        }
     \]
          \[
        v(a_2) = {\scriptstyle
        [11\ 19\ 13\  13\  25\ 36\ 61\ 61\ 61\  61\  122\ 122\ 25\ 39\ 40]^{T},
        }
     \]
      \[
        v(a_3) = {\scriptstyle
        [10\ 10\ 19\  27\ 30\ 40\  70\ 70\ 70\ 70\  140\ 140\ 31\ 49\ 49]^{T}.
        }
     \]
 Therefore, $v_p = \min_{i}\{v_p(f_i)\}$ for $p \in \overline{K}$ is:
   \[
        v = {\scriptstyle
        [10\ 10\ 13\ 13\ 25\ 36\ 61\ 61\ 61\ 61\ 122\ 122\ 25\ 39\ 40]^{T}.
        }
     \]
The corresponding weighted cluster  $\overline{\mathcal{K}}= (\overline{K},v)$ is represented
using the dual graph:

    \tikzstyle{dual}=[circle, draw, fill=black!100, inner sep=0pt, minimum width=4pt]

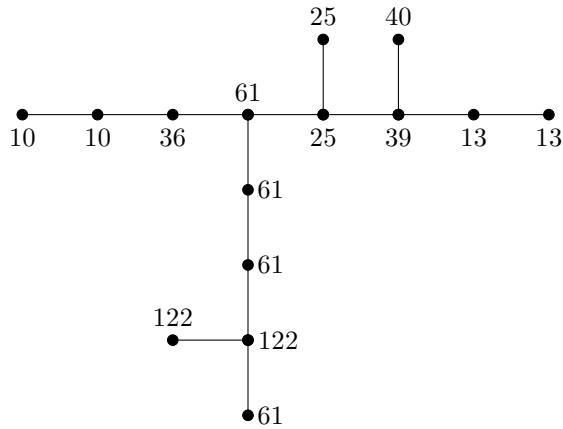
  \begin{center}
    \begin{tikzpicture}
      \draw (-3,0) node[dual]{} -- (-2,0) node[dual]{} -- (-1,0) node[dual] {} -- (0,0) node[dual] {} -- (1,0) node[dual] {} -- (2,0) node[dual] {} -- (3,0) node[dual] {} -- (4,0) node[dual] {};
      \draw (1, 0) node[dual]{} -- (1,1) node[dual]{};
      \draw (2, 0) node[dual]{} -- (2,1) node[dual]{};
      \draw (0, 0) node[dual]{} -- (0,-1) node[dual]{} -- (0,-2) node[dual]{} -- (0,-3) node[dual]{} -- (0,-4) node[dual]{};
      \draw (-1,-3) node[dual]{} -- (0,-3) node[dual]{};

      \draw (-3,-0.3) node {\footnotesize $10$};
      \draw (-2,-0.3) node {\footnotesize $10$};
      \draw (-1,-0.3) node {\footnotesize $36$};
      \draw (0,0.3) node {\footnotesize $61$};
      \draw (1,-0.3) node {\footnotesize $25$};
      \draw (2,-0.3) node {\footnotesize $39$};
      \draw (3,-0.3) node {\footnotesize $13$};
      \draw (4,-0.3) node {\footnotesize $13$};
      \draw (0.3,-1) node {\footnotesize $61$};
      \draw (0.3,-2) node {\footnotesize $61$};
      \draw (0.4,-3) node {\footnotesize $122$};
      \draw (0.3,-4) node {\footnotesize $61$};
      \draw (-1,-2.7) node {\footnotesize $122$};
      \draw (1,1.3) node {\footnotesize $25$};
      \draw (2,1.3) node {\footnotesize $40$};
    \end{tikzpicture}
  \end{center}
\vskip 2mm
  \item There are two missing free base points. The cluster $K'$ is given by the proximity matrix:

    \[
    P_{K'} = \;
  {\tiny
       \begin{bmatrix*}[C]
        1 & 0 & 0 & 0 & 0 & 0 & 0 & 0 & 0 & 0 & 0 & 0 & 0 & 0 & 0 & 0 & 0\\
       -1 & 1 & 0 & 0 & 0 & 0 & 0 & 0 & 0 & 0 & 0 & 0 & 0 & 0 & 0 & 0 & 0\\
       -1 & 0 & 1 & 0 & 0 & 0 & 0 & 0 & 0 & 0 & 0 & 0 & 0 & 0 & 0 & 0 & 0\\
        0 & 0 & -1 & 1 & 0 & 0 & 0 & 0 & 0 & 0 & 0 & 0 & 0 & 0 & 0 & 0 & 0\\
       -1 & 0 & -1 & 0 & 1 & 0 & 0 & 0 & 0 & 0 & 0 & 0 & 0 & 0 & 0 & 0 & 0\\
       -1 & 0 & 0 & 0 & -1 & 1 & 0 & 0 & 0 & 0 & 0 & 0 & 0 & 0 & 0 & 0 & 0\\
        0 & 0 & 0 & 0 & -1 & -1 & 1 & 0 & 0 & 0 & 0 & 0 & 0 & 0 & 0 & 0 & 0\\
        0 & 0 & 0 & 0 & 0 & 0 & -1 & 1 & 0 & 0 & 0 & 0 & 0 & 0 & 0 & 0 & 0\\
        0 & 0 & 0 & 0 & 0 & 0 & 0 & -1 & 1 & 0 & 0 & 0 & 0 & 0 & 0 & 0 & 0\\
        0 & 0 & 0 & 0 & 0 & 0 & 0 & 0 & -1 & 1 & 0 & 0 & 0 & 0 & 0 & 0 & 0\\
        0 & 0 & 0 & 0 & 0 & 0 & 0 & 0 & -1 & -1 & 1 & 0 & 0 & 0 & 0 & 0 & 0\\
        0 & 0 & 0 & 0 & 0 & 0 & 0 & 0 & 0 & 0 & -1 & 1 & 0 & 0 & 0 & 0 & 0\\
        0 & 0 & 0 & 0 & -1 & 0 & 0 & 0 & 0 & 0 & 0 & 0 & 1 & 0 & 0 & 0 & 0\\
        0 & 0 & -1 & 0 & -1 & 0 & 0 & 0 & 0 & 0 & 0 & 0 & 0 & 1 & 0 & 0 & 0\\
        0 & 0 & 0 & 0 & 0 & 0 & 0 & 0 & 0 & 0 & 0 & 0 & 0 & -1 & 1 & 0 & 0\\
        0 & 0 & 0 & 0 & 0 & 0 & 0 & 0 & 0 & 0 & 0 & 0 & 0 & 0 & -1 & 1 & 0\\
        0 & 0 & 0 & 0 & 0 & 0 & 0 & 0 & 0 & 0 & 0 & 0 & 0 & 0 & 0 & -1 & 1\\
       \end{bmatrix*}}.
  \]
The  virtual values are:
     \[
        v(a_1) = {\scriptstyle
        [10\ 10\ 14\ 14\ 28\ 40\ 70\ 72\ 74\ 75\ 150\ 151\ 28\ 42\ 42\ 42\ 42]^{T},
        }
     \]
          \[
        v(a_2) = {\scriptstyle
        [11\ 19\ 13\ 13\ 25\ 36\ 61\ 61\ 61\ 61\ 122\ 122\ 25\ 39\ 40\ 41\ 42]^{T},
        }
     \]
      \[
        v(a_3) = {\scriptstyle
        [10\ 10\ 19\ 27\ 30\ 40\ 70\ 70\ 70\ 70\ 140\ 140\ 31\ 49\ 49\ 49\ 49]^{T}.
        }
     \]
     Thus, we have
   \[
        v = {\scriptstyle
        [10\ 10\ 13\ 13\ 25\ 36\ 61\ 61\ 61\ 61\ 122\ 122\ 25\ 39\ 40\ 41\ 42]^{T},
        }
     \]
and the corresponding weighted cluster $\mathcal{K}' = (K',v)$ is represented by

   \tikzstyle{dual}=[circle, draw, fill=black!100, inner sep=0pt, minimum width=4pt]

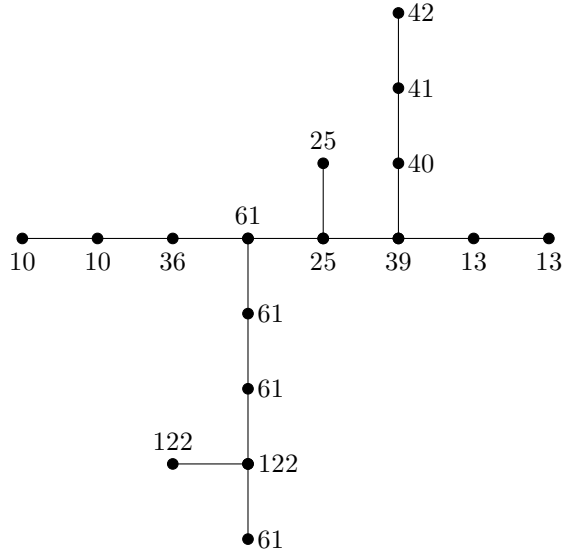
  \begin{center}
    \begin{tikzpicture}
      \draw (-3,0) node[dual]{} -- (-2,0) node[dual]{} -- (-1,0) node[dual] {} -- (0,0) node[dual] {} -- (1,0) node[dual] {} -- (2,0) node[dual] {} -- (3,0) node[dual] {} -- (4,0) node[dual] {};
      \draw (1, 0) node[dual]{} -- (1,1) node[dual]{};
      \draw (2, 0) node[dual]{} -- (2,1) node[dual]{} -- (2,2) node[dual]{} -- (2,3) node[dual]{};
      \draw (0, 0) node[dual]{} -- (0,-1) node[dual]{} -- (0,-2) node[dual]{} -- (0,-3) node[dual]{} -- (0,-4) node[dual]{};
      \draw (-1,-3) node[dual]{} -- (0,-3) node[dual]{};

      \draw (-3,-0.3) node {\footnotesize $10$};
      \draw (-2,-0.3) node {\footnotesize $10$};
      \draw (-1,-0.3) node {\footnotesize $36$};
      \draw (0,0.3) node {\footnotesize $61$};
      \draw (1,-0.3) node {\footnotesize $25$};
      \draw (2,-0.3) node {\footnotesize $39$};
      \draw (3,-0.3) node {\footnotesize $13$};
      \draw (4,-0.3) node {\footnotesize $13$};
      \draw (0.3,-1) node {\footnotesize $61$};
      \draw (0.3,-2) node {\footnotesize $61$};
      \draw (0.4,-3) node {\footnotesize $122$};
      \draw (0.3,-4) node {\footnotesize $61$};
      \draw (-1,-2.7) node {\footnotesize $122$};
      \draw (1,1.3) node {\footnotesize $25$};
      \draw (2.3,1) node {\footnotesize $40$};
      \draw (2.3,2) node {\footnotesize $41$};
      \draw (2.3,3) node {\footnotesize $42$};
    \end{tikzpicture}
  \end{center}
\vskip 2mm

  \item There are four missing satellite base points. The cluster $K''$ is given by:
    \[
    P_{K''} = \;
  {\tiny
       \begin{bmatrix*}[C]
        1 & 0 & 0 & 0 & 0 & 0 & 0 & 0 & 0 & 0 & 0 & 0 & 0 & 0 & 0 & 0 & 0 & 0 & 0 & 0 & 0\\
       -1 & 1 & 0 & 0 & 0 & 0 & 0 & 0 & 0 & 0 & 0 & 0 & 0 & 0 & 0 & 0 & 0 & 0 & 0 & 0 & 0\\
       -1 & 0 & 1 & 0 & 0 & 0 & 0 & 0 & 0 & 0 & 0 & 0 & 0 & 0 & 0 & 0 & 0 & 0 & 0 & 0 & 0\\
        0 & 0 & -1 & 1 & 0 & 0 & 0 & 0 & 0 & 0 & 0 & 0 & 0 & 0 & 0 & 0 & 0 & 0 & 0 & 0 & 0\\
       -1 & 0 & -1 & 0 & 1 & 0 & 0 & 0 & 0 & 0 & 0 & 0 & 0 & 0 & 0 & 0 & 0 & 0 & 0 & 0 & 0\\
       -1 & 0 & 0 & 0 & -1 & 1 & 0 & 0 & 0 & 0 & 0 & 0 & 0 & 0 & 0 & 0 & 0 & 0 & 0 & 0 & 0\\
        0 & 0 & 0 & 0 & -1 & -1 & 1 & 0 & 0 & 0 & 0 & 0 & 0 & 0 & 0 & 0 & 0 & 0 & 0 & 0 & 0\\
        0 & 0 & 0 & 0 & 0 & 0 & -1 & 1 & 0 & 0 & 0 & 0 & 0 & 0 & 0 & 0 & 0 & 0 & 0 & 0 & 0\\
        0 & 0 & 0 & 0 & 0 & 0 & 0 & -1 & 1 & 0 & 0 & 0 & 0 & 0 & 0 & 0 & 0 & 0 & 0 & 0 & 0\\
        0 & 0 & 0 & 0 & 0 & 0 & 0 & 0 & -1 & 1 & 0 & 0 & 0 & 0 & 0 & 0 & 0 & 0 & 0 & 0 & 0\\
        0 & 0 & 0 & 0 & 0 & 0 & 0 & 0 & -1 & -1 & 1 & 0 & 0 & 0 & 0 & 0 & 0 & 0 & 0 & 0 & 0\\
        0 & 0 & 0 & 0 & 0 & 0 & 0 & 0 & 0 & 0 & -1 & 1 & 0 & 0 & 0 & 0 & 0 & 0 & 0 & 0 & 0\\
        0 & 0 & 0 & 0 & -1 & 0 & 0 & 0 & 0 & 0 & 0 & 0 & 1 & 0 & 0 & 0 & 0 & 0 & 0 & 0 & 0\\
        0 & 0 & -1 & 0 & -1 & 0 & 0 & 0 & 0 & 0 & 0 & 0 & 0 & 1 & 0 & 0 & 0 & 0 & 0 & 0 & 0\\
        0 & 0 & 0 & 0 & 0 & 0 & 0 & 0 & 0 & 0 & 0 & 0 & 0 & -1 & 1 & 0 & 0 & 0 & 0 & 0 & 0\\
        0 & 0 & 0 & 0 & 0 & 0 & 0 & 0 & 0 & 0 & 0 & 0 & 0 & 0 & -1 & 1 & 0 & 0 & 0 & 0 & 0\\
        0 & 0 & 0 & 0 & 0 & 0 & 0 & 0 & 0 & 0 & 0 & 0 & 0 & 0 & 0 & -1 & 1 & 0 & 0 & 0 & 0\\
        -1 & 0 & 0 & 0 & 0 & -1 & 0 & 0 & 0 & 0 & 0 & 0 & 0 & 0 & 0 & 0 & 0 & 1 & 0 & 0 & 0\\
        -1 & 0 & 0 & 0 & 0 & 0 & 0 & 0 & 0 & 0 & 0 & 0 & 0 & 0 & 0 & 0 & 0 & -1 & 1 & 0 & 0\\
        -1 & 0 & 0 & 0 & 0 & 0 & 0 & 0 & 0 & 0 & 0 & 0 & 0 & 0 & 0 & 0 & 0 & 0 & -1 & 1 & 0\\
        -1 & 0 & 0 & 0 & 0 & 0 & 0 & 0 & 0 & 0 & 0 & 0 & 0 & 0 & 0 & 0 & 0 & 0 & 0 & -1 & 1\\
       \end{bmatrix*}}.
  \]
The  virtual values are:
     \[
        v(a_1) = {\scriptstyle
        [10\ 10\ 14\ 14\ 28\ 40\ 70\ 72\ 74\ 75\  150\ 151\ 28\ 42\ 42\ 42\ 42\ 50\ 60\ 70\ 80]^{T},
        }
     \]
          \[
        v(a_2) = {\scriptstyle
        [11\ 19\ 13\ 13\ 25\ 36\ 61\ 61\ 61\ 61\  122\ 122\ 25\ 39\ 40\ 41\ 42\ 47\ 58\ 69\ 80]^{T},
        }
     \]
      \[
        v(a_3) = {\scriptstyle
        [10\ 10\ 19\ 27\ 30\ 40\ 70\ 70\ 70\ 70\ 140\ 140\ 31\ 49\ 49\ 49\ 49\ 50\ 60\ 70\ 80]^{T}.
        }
     \]
          Thus, we have:
   \[
        v = {\scriptstyle
        [10\ 10\ 13\ 13\ 25\ 36\ 61\ 61\ 61\ 61\ 122\ 122\ 25\ 39\ 40\ 41\ 42\ 47\ 58\ 69\ 80]^{T},
         } \]
and the corresponding weighted cluster $\mathcal{K}'' = (K'',v)$ is represented by

\tikzstyle{dual}=[circle, draw, fill=black!100, inner sep=0pt, minimum width=4pt]

  \begin{center}
    \begin{tikzpicture}
      \draw  (-7,0) node[dual]{} --  (-6,0) node[dual]{} --  (-5,0) node[dual]{} --  (-4,0) node[dual]{} -- (-3,0) node[dual]{} -- (-2,0) node[dual]{} -- (-1,0) node[dual] {} -- (0,0) node[dual] {} -- (1,0) node[dual] {} -- (2,0) node[dual] {} -- (3,0) node[dual] {} -- (4,0) node[dual] {};
      \draw (1, 0) node[dual]{} -- (1,1) node[dual]{};
      \draw (2, 0) node[dual]{} -- (2,1) node[dual]{} -- (2,2) node[dual]{} -- (2,3) node[dual]{};
      \draw (0, 0) node[dual]{} -- (0,-1) node[dual]{} -- (0,-2) node[dual]{} -- (0,-3) node[dual]{} -- (0,-4) node[dual]{};
      \draw (-1,-3) node[dual]{} -- (0,-3) node[dual]{};

      \draw (-7,-0.3) node {\footnotesize $10$};
      \draw (-6,-0.3) node {\footnotesize $10$};
      \draw (-5,-0.3) node {\footnotesize $47$};
      \draw (-4,-0.3) node {\footnotesize $58$};
      \draw (-3,-0.3) node {\footnotesize $69$};
      \draw (-2,-0.3) node {\footnotesize $80$};
      \draw (-1,-0.3) node {\footnotesize $36$};
      \draw (0,0.3) node {\footnotesize $61$};
      \draw (1,-0.3) node {\footnotesize $25$};
      \draw (2,-0.3) node {\footnotesize $39$};
      \draw (3,-0.3) node {\footnotesize $13$};
      \draw (4,-0.3) node {\footnotesize $13$};
      \draw (0.3,-1) node {\footnotesize $61$};
      \draw (0.3,-2) node {\footnotesize $61$};
      \draw (0.4,-3) node {\footnotesize $122$};
      \draw (0.3,-4) node {\footnotesize $61$};
      \draw (-1,-2.7) node {\footnotesize $122$};
      \draw (1,1.3) node {\footnotesize $25$};
      \draw (2.3,1) node {\footnotesize $40$};
      \draw (2.3,2) node {\footnotesize $41$};
      \draw (2.3,3) node {\footnotesize $42$};
    \end{tikzpicture}
  \end{center}
\vskip 2mm

  \item Using the base change formula $e^T = P_{K''} v^T$ we get
 \[
        e = {\scriptstyle
        [10\  0\  3\  0\  2\  1\  0\  0\  0\  0\  0\  0\  0\  1\  1\  1\  1\  1\  1\  1\  1]^{T}.
         } \]
Thus, erasing the points with multiplicity zero, we finally obtain
the weighted cluster $BP(\fa)=\mathcal{K}=(K,v)$ represented by the proximity matrix
and the vector of values
    \[
    P_{K} = \;
  {\tiny
       \begin{bmatrix*}[C]
              1 & 0 & 0 & 0 & 0 & 0 & 0 & 0 & 0 & 0 & 0 & 0\\
             -1 & 1 & 0 & 0 & 0 & 0 & 0 & 0 & 0 & 0 & 0 & 0\\
             -1 & -1 & 1 & 0 & 0 & 0 & 0 & 0 & 0 & 0 & 0 & 0\\
             -1 & 0 & -1 & 1 & 0 & 0 & 0 & 0 & 0 & 0 & 0 & 0\\
              0 & -1 & -1 & 0 & 1 & 0 & 0 & 0 & 0 & 0 & 0 & 0\\
              0 & 0 & 0 & 0 & -1 & 1 & 0 & 0 & 0 & 0 & 0 & 0\\
              0 & 0 & 0 & 0 & 0 & -1 & 1 & 0 & 0 & 0 & 0 & 0\\
              0 & 0 & 0 & 0 & 0 & 0 & -1 & 1 & 0 & 0 & 0 & 0\\
             -1 & 0 & 0 & -1 & 0 & 0 & 0 & 0 & 1 & 0 & 0 & 0\\
             -1 & 0 & 0 & 0 & 0 & 0 & 0 & 0 & -1 & 1 & 0 & 0\\
             -1 & 0 & 0 & 0 & 0 & 0 & 0 & 0 & 0 & -1 & 1 & 0\\
             -1 & 0 & 0 & 0 & 0 & 0 & 0 & 0 & 0 & 0 & -1 & 1
       \end{bmatrix*},
       \qquad
    v = \begin{bmatrix*}
        10\\
        13\\
        25\\
        36\\
        39\\
        40\\
        41\\
        42\\
        47\\
        58\\
        69\\
        80
        \end{bmatrix*}
  }
     \]
  or equivalently, by the dual graph:

   \tikzstyle{dual}=[circle, draw, fill=black!100, inner sep=0pt, minimum width=4pt]

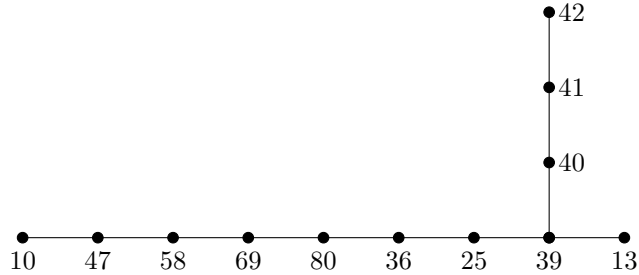
 \begin{center}
    \begin{tikzpicture}
      \draw  (-5,0) node[dual]{} --  (-4,0) node[dual]{} --  (-3,0) node[dual]{} -- (-2,0) node[dual]{} -- (-1,0) node[dual]{} -- (0,0) node[dual] {} -- (1,0) node[dual] {} -- (2,0) node[dual] {} -- (3,0) node[dual] {};
      \draw (2, 0) node[dual]{} -- (2,1) node[dual]{} -- (2,2) node[dual]{} -- (2,3) node[dual]{};

      \draw (-5,-0.3) node {\footnotesize $10$};
      \draw (-4,-0.3) node {\footnotesize $47$};
      \draw (-3,-0.3) node {\footnotesize $58$};
      \draw (-2,-0.3) node {\footnotesize $69$};
      \draw (-1,-0.3) node {\footnotesize $80$};
      \draw (0,-0.3) node {\footnotesize $36$};
      \draw (1,-0.3) node {\footnotesize $25$};
      \draw (2,-0.3) node {\footnotesize $39$};
      \draw (3,-0.3) node {\footnotesize $13$};
      \draw (2.3,1) node {\footnotesize $40$};
      \draw (2.3,2) node {\footnotesize $41$};
      \draw (2.3,3) node {\footnotesize $42$};
    \end{tikzpicture}
  \end{center}
\vskip 2mm
\end{enumerate}
\end{example}

\section{Newton-Puiseux revisited} \label{sec5}

If we take a closer look at Algorithm \ref{A1} we see that all the steps
can be effectively computed once we have a precise description of the weighted cluster
from step ${\rm ii)}$.  The aim of this section is to provide an algorithm
that solves the following problem:

\vskip 2mm

{\it  Given a set of elements $f_1,\dots, f_r \in \bC\{x,y\}$, provide a method to
compute the weighted cluster $\overline{K}$ associated
to the reduced germ  of $\xi: f = f_1 \cdots f_r = 0$ and the systems of virtual values
$\{v_p(f_i)\}_{p \in \overline{K}}$, for $i = 1, \dots, r$.}

\vskip 2mm

We point out that in the case that $f$ is already reduced, we can compute the cluster
$\overline{K}$ using the Newton-Puiseux algorithm and Enriques' theorem \cite[\S1 and \S5.5]{Cas00}.
Computer algebra systems such as {\tt Singular} \cite{DGPS} or {\tt Magma} \cite{Magma} can do the job.
However, we are in a more general situation that requires some extra work.
The Puiseux factorization theorem \cite[\S 1.5]{Cas00} states that any
$g \in \mathbb{C}\{x, y\}$ can be decomposed as
\begin{equation} \label{eq:puiseux-factorization}
g(x, y) = u x^{\alpha_0} g_1^{\alpha_1} \cdots g_\ell^{\alpha_\ell} = u x^{\alpha_0} \prod_{i = 1}^\ell \prod_{j = 1}^{\nu_i} ( y - \sigma_i^j(s_{i}) )^{\alpha_i}, \quad \alpha_1, \dots, \alpha_l \in \mathbb{N}
\end{equation}
where $u \in \mathbb{C}\{x, y\}$ is a unit, $g_1, \dots, g_\ell \in \mathbb{C}\{x, y\}$
are irreducible, $s_{i} \in \mathbb{C}\langle\langle x \rangle\rangle$ are Puiseux series
such that $g_i(x, s_{i}(x)) = 0$, $\nu_i = \textrm{ord}_y(g_i(0, y))$, and $\sigma_i^j$ is the
automorphism of $\mathbb{C}((x^{1/\nu_i}))$ generated by $x^{1/\nu_i} \mapsto e^{2\pi \sqrt{-1} j/\nu_i} x^{1/\nu_i}$.

\vskip 2mm

From the above factorization one can compute the required cluster of singular points and  systems of
virtual values. It is a classical result, see \cite[\S 5.5]{Cas00}, that the Puiseux series $s_i, i = 1, \dots, \ell$
completely determine the cluster of singular points of $\eta_{\rm red}$, where $\eta: g = 0$. In order to compute the virtual
values $v_p(g)$ for any singular point $p$ of $\eta_{\rm red}$ we can use the fact that $v_p$ are valuations, thus

\begin{equation} \label{eq:alg_mult}
v_p(g) = \alpha_{0} v_p(x) + \alpha_{1} v_p(g_{1}) + \cdots + \alpha_{\ell} v_p(g_{\ell}).
\end{equation}

\vskip 2mm

\noindent In addition, the values $v_p(x), v_p(g_{i}), i = 1, \dots, \ell$, can also be deduced from their
associated Puiseux series $s_i$ and the cluster of singular points of $\eta_{\rm red}$. Notice that the algebraic
multiplicities $\alpha_i$ play their role in Equation \ref{eq:alg_mult}.

\vskip 2mm

The Newton-Puiseux algorithm, that traditionally has been used to obtain Puiseux decompositions, only works
for reduced elements. This means that you cannot recover the algebraic multiplicities of the Puiseux
series in Equation \ref{eq:puiseux-factorization}. Another problem that arises when applying the Newton-Puiseux
algorithm to a product $f = f_1 \cdots f_r$ is that you cannot find which factor $f_i$ contains each resulting Puiseux series.

\vskip 2mm

To overcome such inconvenients, we will present a modified version of the Newton-Puiseux
algorithm that, given a set of elements $f_1, \dots, f_r \in \mathbb{C}\{x, y\}$ not
necessarily reduced or irreducible, will compute the Puiseux decomposition of the product
$f = f_1 \cdots f_r$, that is, the Puiseux series of $f$ together with their algebraic
multiplicities in each of the factors $f_1, \dots, f_r$.

\vskip 2mm

The Newton-Puiseux algorithm is obviously restricted to compute a partial sum of each
Puiseux series in the decomposition \ref{eq:puiseux-factorization} as the series are potentially infinite. Thus, the
algorithm computes enough terms of each series so they do not share terms from a certain
degree onward. In this situation we will say that the series have been pair-wise \emph{separated}.
In particular, this means that a partial sum of Puiseux series  $s$ might be enough to separate
$s$ inside a factor, but not inside the whole product $f = f_1 \cdots f_r$. Hence, applying the
Newton-Puiseux algorithm to the factors $f_1, \dots, f_r$ does not provide as much information as
applying the Newton-Puiseux algorithm to the product. Similarly, if one obtains just the Puiseux
series of the product it is not possible to recover the Puiseux decomposition of each factor.
The modification of the Newton-Puiseux algorithm that we will present  provides all the information
needed to recover both the decomposition of each factors and the decomposition of the whole product at the same time.
One of the key ingredients is the \emph{square-free factorization}.

\begin{definition}
Let $R$ be a unique factorization domain. The \emph{square-free factorization} of an element $h \in R[[x]]$ is
\begin{equation} \label{eq:sqf-factorization}
h = h_1 h_2^2 \cdots h_n^n,
\end{equation}
such that $h_i \in R[[x]], i = 1, \dots, n$ are reduced, pair-wise coprime elements, and $h_n$ is a non-unit.
\end{definition}

Notice that some of the $h_i, i = 1, \dots, n-1$ can be units. The non-unit factors
in Equation \ref{eq:sqf-factorization} will be called \emph{square-free factors} and are unique up to
multiplication by a unit.

\vskip 2mm

We will not explain all the details for the traditional Newton-Puiseux algorithm, for that we refer the reader
 to \cite[1.5]{Cas00}. We will just recall that it is an iterative algorithm
that at the $i$-th step computes the $i$-th term of one of the Puiseux series $s$. The first term of
$s(x) =: s^{(0)}(x_0)$ is computed from $f(x, y) =: f^{(0)}(x_0, y_0)$, and the
$i$-th term of $s$ is computed as the first term of
$s^{(i)}(x_i) \in \mathbb{C}\langle\langle x_i \rangle\rangle$ from
$f^{(i)}(x_i, y_i) \in \mathbb{C}\{x_i, y_i\}$ which are defined recursively
from $s^{(i-1)}$ and $f^{(i-1)}(x_{i-1}, y_{i-1})$ by means of a change of variables.

\vskip 2mm

The basic idea behind our new algorithm is to apply the traditional
Newton-Puiseux algorithm to the reduced part of $f$, $\bar{f}$, while the square-free
factors of each $f_i, i = 1, \dots, r$ are transformed using the changes of variables given
by $\bar{f}$. The Newton-Puiseux algorithm applied on $\bar{f}$ will tell when all the branches have been separated,
i.e. the stopping condition. The square-free factors will encode, at the end, the algebraic
multiplicities of the resulting Puiseux series in each factor.

\vspace{0.5cm}

The modified Newton-Puiseux algorithm works as follows:

\vskip 2mm

\begin{itemize}

\item Compute the element $f = f_1 \cdots f_r$ and
$\bar{f} = f/\gcd(f, \frac{\partial f}{\partial y}, \frac{\partial f}{\partial x})$.
Define $x_0 := x, y_0 := y, f^{(0)} := \bar{f}$, and
\begin{align}
S^{(0)} := \{ h_{j, k} \in \mathbb{C}\{x_0, y_0\}\ | & \ h_{j, k} \textrm{ square-free factor of } f_k, k = 1, \dots, r \nonumber \\
  & \textrm{ with multiplicity}\ j \in \mathbb{N}\}
\end{align}

\item \textbf{Step (i)}: The $i$-th iteration runs as in the traditional algorithm and we
compute $x_{i+1}, y_{i+1}$ and $f^{(i+1)}$. In addition, we compute $S^{(i+1)}$ from $S^{(i)}$ in the following way:
\begin{align}
S^{(i+1)} = \{ x_{i+1}^{-\beta_{i, j, k}}h_{j,k}^{(i)}(x_{i+1}, y_{i+1}) \in \mathbb{C}\{x_{i+1}, y_{i+1}\}\ |\
 & x_{i+1}^{\beta_{i, j, k} + 1} \not |\ h_{j, k}^{(i)}(x_{i+1}, y_{i+1}), \nonumber \\
 & h_{j, k}^{(i+1)} \textrm{ non-unit}, h_{j, k}^{(i)} \in S^{(i)} \} \nonumber
\end{align}

\item The algorithm ends at the same step the traditional Newton-Puiseux algorithm ends for
the reduced part $\bar{f}$.
\end{itemize}

\vspace{0.5cm}

In order to prove the correctness of this modification we will need the following results.

\begin{lemma}[{\cite[1.6.3]{Cas00}}] \label{lemma:puiseux-mult}
For any $j > i \ge 0$, the multiplicity of $s^{(i)}$ as Puiseux series of $f^{(i)}$ equals the multiplicity
of $s^{(j)}$ as Puiseux series of $f^{(j)}$.
\end{lemma}

In the current context the following lemma follows from the definitions.

\begin{lemma} \label{lemma:coprimes}
Two elements of $\mathbb{C}\{x, y\}$ are coprime if and only if they share no Puiseux series and no $x$ factor.
\end{lemma}

\begin{proposition} \label{prop:sqf-factors}
The set $S^{(i)}$ contains the square-free factors of $f^{(i)}_k$ for any $i \ge 0$ and any $k = 1, \dots, r$.
\end{proposition}

\begin{proof}
By induction on $i \ge 0$. By construction, $S^{(0)}$ contains the square-free factors of
$f^{(0)}_k := f^k$, for $k = 1, \dots, r$. Assume now that $S^{(i)}$ contains the square-free
factors of $f^{(i)}_k$.

If two elements of $h^{(i+1)}_{n, k}, h^{(i+1)}_{m, k}$ are not coprime, they would share a Puiseux
series or an $x$ factor, by Lemma \ref{lemma:coprimes}. The $x$ factor is not possible by definition of
$S^{(i+1)}$. If they share a Puiseux series $s^{(i+1)}$, $s^{(i)}$ would be a series of
$h^{(i)}_{n, k}$ and $h^{(i)}_{m, k}$, contradicting the induction hypothesis. Since $h^{(i)}_{j, k}$ is
reduced so is $h^{(i+1)}_{j, k}$, by Lemma \ref{lemma:puiseux-mult}. Since Equation \ref{eq:sqf-factorization} still holds
after applying the change of variables two both sides, the result follows.
\end{proof}

\begin{proposition} \label{prop:last-proposition}
Assume $s \in \mathbb{C}\langle\langle x \rangle\rangle$ has been separated from
the rest of the series of $\bar{f}$ at the $i$-th step of the algorithm. Then, $s$ is a
Puiseux series of $f_k \in \mathbb{C}\{x, y\}$ with algebraic multiplicity $j \in \mathbb{N}$ if and
only if $h^{(i)}_{j, k} \in S^{(i)}$.
\end{proposition}
\begin{proof}
For the direct implication, assume that $s$ is a Puiseux series of $f_k$ with multiplicity $j \in \mathbb{N}$. Then,
$s$ is a Puiseux series of $h^{(0)}_{j, k} \in S^{(0)}$ and no other square-free factor,
by Lemma \ref{lemma:coprimes}. Now, by Lemma \ref{lemma:puiseux-mult}, $s^{(i)}$ is a root of $h^{(i)}_{j, k}$ and
it belongs to $S^{(i)}$ because it is a non-unit. For the converse, since $s$ has been separated, $f^{(i)}$ has no other
Puiseux series other than $s^{(i)}$ and its conjugates. By Proposition \ref{prop:sqf-factors}, there must be a
unique $h^{(i)}_{j, k}$ square-free factor of $f_k^{(i)}$ in $S^{(i)}$. Finally, by Lemma \ref{lemma:puiseux-mult}, if
the algebraic multiplicity of $s^{(i)}$ is $j > 0$ in $f^{(i)}_k$, so is the algebraic multiplicity of $s$ in $f_k$.
\end{proof}

It follows from Proposition \ref{prop:last-proposition} that, when the algorithm
stops at the $i$-th step after $s$ has been separated, the set $S^{(i)}$ contains the information about
the factors and the algebraic multiplicities of the Puiseux series $s$.

\subsection{Implementation details}

The algorithms discussed in this paper have been implemented in the computer algebra systems
{\tt Macaulay2} \cite{M2} and {\tt Magma} \cite{Magma} and they are available at
\begin{center}
\url{https://github.com/gblanco92/}.
\end{center}

\vskip 2mm

The {\tt Macaulay2} implementation uses floating point arithmetic
to compute the Puiseux series. This could potentially give inaccurate computations and wrong
results so we implemented the same algorithms in {\tt Magma} using algebraic field extensions. So far, all the
examples the authors have tested give the same result in both implementations.

\vskip 2mm

As usual when developing algorithms in computer algebra, we have to work with polynomials
in $\mathbb{C}[x, y]$ instead of series in $\mathbb{C}\{x, y\}$. If we take a close look at all
the steps of Algorithm \ref{A1} and the new formulation of the Newton-Puiseux algorithm we see that
this is not an issue. Indeed, it is not a problem for the traditional Newton-Puiseux algorithm.
Also, the square-free decomposition of elements of $\mathbb{C}[x, y]$ is an standard tool in computer algebra
and can be computed efficiently, see for instance \cite{Yun76}. We also point out that, given a reduced polynomial
$f \in \mathbb{C}[x, y]$, it remains reduced when viewed in $\mathbb{C}\{x, y\}$ (see \cite{chevalley}). Moreover,
a greatest common divisor in the polynomial ring is still a greatest common divisor in the convergent series ring.

\vskip 2mm

The fact that the new Newton-Puiseux algorithm works with the square-free factors of the generators,
which are reduced and generally of smaller degree than the original polynomials, means that the computation
can be kept efficient. Working with non-reduced elements would increase significantly the computational time
of the Newton-Puiseux algorithm.

\vskip 2mm

Finally, we would like to mention that Algorithm \ref{A1} is one of the key
ingredients of the method we develop in \cite{ACAMB16-2} to compute the
integral closure of any ideal $\fa \subseteq \mathbb{C}\{x, y\}$.
We hope that these algorithms can be useful to people interested in the
computational aspects of singularity theory. For example, our methods
are very helpful in the effective computation
of {\it multiplier ideals} (see \cite{ACAMDC} and \cite{BDC16}).

\end{document}